\documentclass[12pt, a4paper]{amsart}
\usepackage[T2A]{fontenc}
\usepackage{amsmath}
\usepackage{extarrows}
\usepackage{amssymb,amsthm,amscd}
\usepackage{graphicx}
\usepackage[lmargin=3cm,rmargin=3cm,tmargin=2cm,bmargin=2cm, marginpar=2cm ]{geometry}

\usepackage[shortlabels]{enumitem}

\usepackage{hyperref} % loaded late (after geometry/enumitem), before refcheck and todonotes
\usepackage[obeyDraft]{todonotes} % https://tug.ctan.org/macros/latex/contrib/todonotes/todonotes.pdf

\theoremstyle{plain}
\newtheorem{theorem}{Theorem}[section]
\newtheorem{lemma}[theorem]{Lemma}
\newtheorem{conjecture}[theorem]{Conjecture}

\newtheorem{proposition}[theorem]{Proposition}

\theoremstyle{definition}
\newtheorem{definition}[theorem]{Definition}
\newtheorem{example}[theorem]{Example}

\newtheorem{notation}[theorem]{Notation}

\theoremstyle{remark}

\def\Aut{\operatorname{Aut}}
\def\Inn{\operatorname{Inn}}

\def\LND{\operatorname{LND}}

\def\Ker{\operatorname{Ker}}% \stackrel{\sim}{\longrightarrow} ?

\def\PP{{\mathbb P}}

\def\AA{{\mathbb A}}

\def\KK{{\mathbb K}}

\def\Ga{\mathbb{G}_{\rm a}}

\def\embed{\hookrightarrow}

\def\0{\circ}

\title[Automorphism groups of non-normal rigid affine surfaces]{Automorphism groups of non-normal rigid affine surfaces are finite-dimensional}

\author[I. Beldiev]{Ivan Beldiev}
\email{ivbeldiev@gmail.com, isbeldiev@hse.ru}
\address{HSE University, Faculty of Computer Science,
Pokrovsky blvd. 11, Moscow, 109028 Russia}

\author[A. Perepechko]{Alexander Perepechko} 
\email{a@perep.ru}
\address{
HSE University, Faculty of Computer Science,
Pokrovsky blvd. 11, Moscow, 109028 Russia} 

\thanks{This work was supported by the Theoretical Physics and Mathematics Advancement Foundation ``BASIS''}



\begin{document}

\begin{abstract}
  It was recently established by Perepechko and Zaidenberg that the automorphism group of a normal affine surface is finite-dimensional if and only if the surface admits no non-trivial action of the additive group of the base field. We extend this result to non-normal affine surfaces.
\end{abstract}
%\date{\today}

\maketitle

%%%%%%%%%%%%%%%%%%%%%%%%%%%%%%%%%%%%%%%%%%%%%%%%%%%%%%%%%%%%%%%%%%%%%%
\section{Introduction}\label{sec:intro}
We work over an algebraically closed field $\KK$ of characteristic zero.
By an \emph{affine variety}, we mean an integral affine scheme of finite type over $\KK$; an \emph{affine surface} is an affine variety of dimension~$2$. We denote by $\Ga$ the additive group $(\KK,+)$ of the base field.

Let $Y$ be an affine variety. The group $\Aut(Y)$ of regular automorphisms of $Y$ carries a natural structure of an ind-group; see \cite[Theorem~5.1.1]{FuKr18}. In particular, its \emph{neutral component} $\Aut^\circ(Y)$, i.e., the union of all irreducible algebraic subsets of $\Aut(Y)$ containing the identity, is well defined. In contrast to the case of a projective variety, the ind-group $\Aut(Y)$ is in general infinite-dimensional; in this case the neutral component $\Aut^\circ(Y)$ is not an algebraic group.

The question whether $\Aut^\circ(Y)$ is algebraic is related to the existence of $\Ga$-actions on $Y$, i.e., regular actions of the algebraic group $\Ga$ on $Y$. 
It is known (see \cite[Corollary~1.2]{FlZa05}) that $\Aut^\circ(Y)$ is not algebraic if $Y$ admits at least one non-trivial $\Ga$-action and $\dim Y \geq 2$. 
Conversely, suppose that $Y$ admits no non-trivial $\Ga$-actions; such varieties $Y$ are called \emph{rigid}. 
It is conjectured that in this case $\Aut^\circ(Y)$ is algebraic; cf.~\cite{PeRe23, PeZa26}. 
Indeed, if $Y \ncong \AA^1$, then $\Aut^\circ(Y)$ is algebraic if and only if it is an algebraic torus; see \cite[Proposition~5]{Ii77} and \cite[Theorem~1.3]{Kr17}. 
The precise statement of the conjecture is the following.

\begin{conjecture}\label{conj1}
    Let $Y$ be an affine algebraic variety over $\KK$. If $Y$ is rigid, then the group $\Aut^\circ(Y)$ is an algebraic torus of rank at most $\dim Y$.
\end{conjecture}

If Conjecture~\ref{conj1} holds, then so does the following one; cf. the list of open problems in~\cite{Kr25}.

\begin{conjecture}\label{conj2}
    If an affine variety $Y$ admits no effective $\Ga$- or $\mathbb G_m$-actions, then $\Aut(Y)$ is a discrete group.
\end{conjecture}

% The equivalence of conditions for $\Aut^\circ(Y)$ to be algebraic and an algebraic torus in the case $Y\ncong \AA^1$ is covered, e.g., in \cite[Proposition~5]{Ii77} and \cite[Theorem~1.3]{Kr17}.
% Some other partial results can be found in \cite{PeRe23, PeRe24}. 

% Conjecture~\ref{conj1} also holds, for instance, for rigid trinomial hypersurfaces (\cite[Theorem~3]{ArGa17}), for toric affine varieties (\cite[Theorem~3]{BoGa21}), for rational affine varieties with a torus action of complexity one (\cite[Theorem~6.4]{BorGa25}), and .

% In \cite{PeZa26}, the authors prove Conjecture~\ref{conj1} for normal affine surfaces. The proof involves the notions of a normal-crossing completion $(\widetilde{Y}, D)$ of a normal affine surface $Y$ and the associated weighted dual graph $\Gamma(D)$; these notions are recalled in Section~\ref{sec:NC}. The precise result proved in \cite{PeZa26} is the following theorem.

Conjecture~\ref{conj1} is known to hold, for instance, for rigid trinomial hypersurfaces (\cite[Theorem~3]{ArGa17}), for toric affine varieties (\cite[Theorem~3]{BoGa21}), and for rational affine varieties with a torus action of complexity one (\cite[Theorem~6.4]{BorGa25}). Some partial results in the general setting are obtained in \cite{PeRe23, PeRe24}.

For normal affine surfaces, the conjecture is settled in \cite{PeZa26}, as follows.

\begin{theorem}[{\cite[Theorem~1.0.3(a)]{PeZa26}}]\label{th:intro-PZ}
Let $Y$ be a normal affine surface over $\KK$. Then the neutral component of the automorphism group $\Aut^{\circ}(Y)$ is an algebraic group if and only if $Y$ admits no effective $\Ga$-action. In this case, $\Aut^{\circ}(Y)$ is an algebraic torus of rank $\le 2$.
\end{theorem}

In the present paper we remove the normality assumption, establishing Conjecture~\ref{conj1} for all affine surfaces. 

\begin{theorem}\label{main_result}
    Let $Y$ be a rigid, not necessarily normal, affine surface. Then the neutral component $\Aut^\circ(Y)$ is an algebraic torus of rank $\le 2$.
\end{theorem}

This completes the proof of Conjecture~\ref{conj1} in the case $\dim Y = 2$. 
The main ingredient of our proof is to pass to the normalization $\widehat{Y}$ of a non-normal affine surface $Y$ and to apply the birational techniques of \cite{PeZa26} to the normal affine surface~$\widehat{Y}$.

The paper is organized as follows. In Section~\ref{sec:nonnormal} we pass to the normalization $\pi\colon\widehat Y\to Y$: we embed $\Aut(Y)$ into the group $\Aut(\widehat Y,\widehat E)$ of automorphisms preserving the exceptional set (Proposition~\ref{embed}) and describe $\widehat E$ as the vanishing set of the conductor ideal (Proposition~\ref{cond_non_normal}). In Section~\ref{sec:Ga-lift} we show that, although an $\widehat E$-preserving $\Ga$-action on $\widehat Y$ need not descend to $Y$ (Example~\ref{counterexample}), rigidity of $Y$ nonetheless forbids such actions: $\widehat Y$ admits no non-trivial $\Ga$-action preserving $\widehat E$ (Proposition~\ref{sending_down}). In Section~\ref{sec:NC} we recall normal-crossing completions and the rigidity criterion of~\cite{PeZa26} (Theorem~\ref{th:PZ-NC}) and adapt its proof to a relative setting: for a proper closed subset $E$ of a normal affine surface $X$, if $\Aut(X,E)$ contains no $\Ga$-subgroup, then $\Aut^\circ(X,E)$ is algebraic (Proposition~\ref{prop:PZ-variation}). Applying this with $(X,E)=(\widehat Y,\widehat E)$ in Section~\ref{sec:main} shows that $\Aut^\circ(\widehat Y,\widehat E)$ is algebraic; hence so is $\Aut^\circ(Y)$ via Proposition~\ref{embed}, and rigidity forces it to be an algebraic torus (Theorem~\ref{main_result}).

\section{Normalization}\label{sec:nonnormal}

Let $Y$ be a non-normal affine variety. 
Recall that its \emph{normalization} is the unique normal affine variety $\widehat{Y}$ together with a finite birational morphism $\pi\colon\widehat{Y}\to Y$ through which every dominant morphism from a normal variety to $Y$ factors uniquely.  
We denote by $E\subseteq Y$ and $\widehat{E}\subseteq\widehat{Y}$ the \emph{exceptional sets} of $Y$ and $\widehat{Y}$ under $\pi$, respectively:
they are the smallest closed subsets such that $\pi\colon \widehat{Y}\setminus\widehat{E} \to Y\setminus E$ is an isomorphism.
In particular, $E$ is the closed set of non-normal points of $Y$, and $\widehat{E}$ is its preimage under $\pi$.
Algebraically, the algebra of regular functions $\KK[\widehat Y]$ is the integral closure of $\KK[Y]$ in the field of rational functions $\KK(Y)$, and the dual homomorphism $\pi^*$ is the inclusion $\KK[Y]\hookrightarrow \KK[\widehat Y]$.

In Proposition~\ref{embed}, we relate the automorphisms of $Y$ to those of its normalization.
Namely, we show that $\Aut(Y)$ embeds into the subgroup of automorphisms of $\widehat Y$ preserving $\widehat E$.
The following lemma seems to be well known; we state and prove it here for completeness.

\begin{lemma}\label{normal_aut}
    Every automorphism $\varphi$ of $Y$ can be lifted to a unique automorphism $\widehat \varphi$ of $\widehat Y$ such that $\pi \circ \widehat \varphi = \varphi \circ \pi$. Moreover, this lifting is functorial, i.e., $\widehat{\mathrm{id}_Y} = \mathrm{id}_{\widehat Y}$ and $\widehat{\varphi\circ \psi } = \widehat \varphi \circ \widehat \psi$ for any two automorphisms $\varphi$ and $\psi$ of $Y$.
\end{lemma}

\begin{proof}
    By the universal property of the normalization, for any normal affine variety $Z$ and any dominant morphism $f \colon Z \to Y$ there exists a unique morphism $g\colon Z \to \widehat Y$ such that $f = \pi\circ g$. Applying this to $Z = \widehat Y$ and $f = \varphi \circ \pi$, we obtain a unique morphism $\widehat \varphi\colon \widehat Y \to \widehat Y$ satisfying $\pi\circ\widehat\varphi = \varphi\circ\pi$. The uniqueness part of the universal property also yields $\widehat{\mathrm{id}_Y} = \mathrm{id}_{\widehat Y}$ and $\widehat{\varphi\circ \psi } = \widehat \varphi \circ \widehat \psi$ for any two automorphisms $\varphi$ and $\psi$ of $Y$. Applying the latter to $\psi = \varphi^{-1}$, we see that $\mathrm{id}_{\widehat Y} = \widehat{\mathrm{id}_Y} = \widehat{\varphi\circ\varphi^{-1}} = \widehat{\varphi}\circ \widehat{\varphi^{-1}}$, so $\widehat{\varphi}$ is indeed an automorphism of~$\widehat Y$.
\end{proof}

\begin{proposition}\label{embed}
    The group $\Aut(Y)$ of automorphisms of $Y$ is naturally embedded into the group $\Aut(\widehat Y, \widehat E)$ of automorphisms of $\widehat Y$ preserving the exceptional set $\widehat E$ of the normalization.
\end{proposition}

\begin{proof}
    By Lemma~\ref{normal_aut}, the map $\varphi \mapsto \widehat \varphi$ embeds the group $\Aut(Y)$ into the group $\Aut(\widehat Y)$ of all automorphisms of $\widehat Y$. Moreover, the exceptional set $E$ of non-normal points of $Y$ is invariant under any automorphism $\varphi \in \Aut(Y)$, hence $\widehat \varphi$ preserves the set $\widehat E = \pi^{-1}(E)$. 
\end{proof}

In Proposition~\ref{cond_non_normal}, we characterize the exceptional set as the vanishing set of the \emph{conductor ideal}.
\begin{definition}
    Let $R$ be an integral domain and let $\mathcal O_R$ be the integral closure of $R$ in the field of fractions $\operatorname{Frac}(R)$. The \emph{conductor} of $R$ is the ideal
    \[C_R = \{r\in R\mid r\mathcal O_R \subseteq R\}\subseteq R.\]
\end{definition}
It is an ideal in both $R$ and $\mathcal O_R$, and it is the largest ideal of $\mathcal O_R$ contained in $R$.

\begin{proposition}\label{cond_non_normal}
    Let $C_Y\subseteq \KK[Y]$ be the conductor ideal of $\KK[Y]$. Then its vanishing set \[\mathbb{V}(C_Y) := \{y\in Y\mid f(y) = 0 \quad \forall f \in C_Y\}\subseteq Y\] coincides with the exceptional set $E$ of the normalization.
\end{proposition}

\begin{proof}
    Since $\pi$ is birational, we identify $\KK[\widehat Y]$ with a subring of the function field $\KK(Y)$, so that
    \[
        \KK[Y]\subseteq\KK[\widehat Y]\subseteq\KK(Y),
    \]
    and all local rings of points of $Y$ and of $\widehat Y$ are subrings of $\KK(Y)$ as well. Recall that $y\notin E$ if and only if the local ring $\mathcal O_y(Y)$ is integrally closed.

    First we prove the inclusion $E \subseteq \mathbb{V}(C_Y)$. Let $y\in Y\setminus \mathbb{V}(C_Y)$ and choose $f\in C_Y$ with $f(y)\ne 0$. For every $g\in\KK[\widehat Y]$ we have $fg\in\KK[Y]$ by the definition of $C_Y$; since $f$ is invertible in $\mathcal O_y(Y)$, it follows that $g = \frac{fg}{f}\in\mathcal O_y(Y)$. Hence $\KK[\widehat Y]\subseteq\mathcal O_y(Y)$. Now let $S = \{s\in\KK[Y]\mid s(y)\ne 0\}$, so that $\mathcal O_y(Y) = S^{-1}\KK[Y]$. We obtain
    \[
        \mathcal O_y(Y) = S^{-1}\KK[Y] \subseteq S^{-1}\KK[\widehat Y] \subseteq \mathcal O_y(Y),
    \]
    hence $\mathcal O_y(Y) = S^{-1}\KK[\widehat Y]$. Thus $\mathcal O_y(Y)$ is a localization of the integrally closed domain $\KK[\widehat Y]$ and is therefore integrally closed itself, i.e., $y\notin E$.

    Now we prove the inclusion $\mathbb{V}(C_Y) \subseteq E$. Let $y\notin E$. Since $E$ is closed and $y\notin E$, there exists $f\in\KK[Y]$ such that $f$ vanishes on $E$ and $f(y)\ne 0$. Then the principal open subset
    \[
        D(f) = \{x \in Y \mid f(x) \ne 0\}
    \]
    is contained in $Y\setminus E$, over which $\pi$ is an isomorphism.
    Therefore
    \[
        \KK[\widehat Y] \subseteq  \KK[Y]_f.
    \]

    Since $\pi$ is finite, $\KK[\widehat Y]$ is a finitely generated $\KK[Y]$-module; let $g_1,\ldots,g_k\in\KK[\widehat Y]$ be a set of generators. By the above, $g_i = \frac{h_i}{f^{n_i}}$ in $\KK(Y)$ for some $h_i\in\KK[Y]$ and non-negative integers $n_i$. Put $n = \max_i n_i$ and $f' = f^n$. Then $f'g_i = f^{\,n-n_i}h_i\in\KK[Y]$ for every $i$, and hence $f'g\in\KK[Y]$ for every $g = \sum_i a_ig_i\in\KK[\widehat Y]$ with $a_i\in\KK[Y]$. Thus $f'\in C_Y$, while $f'(y) = f(y)^n\ne 0$, so $y\notin\mathbb{V}(C_Y)$.
\end{proof}

\section{Lifting $\Ga$-actions}\label{sec:Ga-lift}

 A \emph{$\Ga$-action} on an affine variety $Y$ is a regular group homomorphism $\Ga\to\Aut(Y)$.  For an affine variety $Y$, such actions correspond bijectively to \emph{locally nilpotent derivations} (LNDs) of $\KK[Y]$, i.e., $\KK$-derivations $\partial\colon\KK[Y]\to\KK[Y]$ such that for every $f\in\KK[Y]$ there exists $n\ge 0$ with $\partial^n(f)=0$.  
 The bijection sends $\partial$ to the action given on functions by $(t,f)\mapsto\exp(t\partial)(f)$; see~\cite{Fr06}. 
 We write $\LND(\KK[Y])$ for the set of all LNDs of~$\KK[Y]$.
An affine variety $Y$ is called \emph{rigid} if $\LND(\KK[Y])=0$, i.e.\ if $Y$ admits no
non-trivial $\Ga$-action; see~\cite{ArGa17}.

We now assume that $Y$ is a rigid non-normal affine surface and keep the notation of Section~\ref{sec:nonnormal} for its normalization. 
It can happen that the normalization $\widehat Y$ is not itself rigid, i.e., admits a non-trivial $\Ga$-action.
Moreover, a $\Ga$-action on $\widehat Y$ preserving $\widehat E$ need not descend to a $\Ga$-action on $Y$ via the embedding of Proposition~\ref{embed}; we illustrate this explicitly in Example~\ref{counterexample}.
Nevertheless, rigidity of $Y$ implies the absence of non-trivial $\Ga$-actions in $\Aut(\widehat Y,\widehat E)$, see Proposition~\ref{sending_down}.

\begin{example}\label{counterexample}
    Let $Y\subseteq \AA^3$ be the surface given by the equation $x^3 - y^3t = 0$. The algebra $\KK[Y]$ of regular functions on $Y$ is the quotient $\KK[x,y,t]/(x^3 - y^3t)$. Since the element $q:=x/y\in\KK(Y)$ satisfies the equation $q^3 - t = 0$, the intergral closure of $\KK[Y]$ in $\KK(Y)$ is the polynomial algebra $\KK[q,y]$. Hence the normalization of $Y$ is the affine plane $\AA^2$ with coordinates $(q,y)$, and the normalization map $\pi\colon \AA^2 \to Y$ is given by 
    \[\pi:(q,y) \mapsto (qy,\, y,\, q^3).\]
    We see that the inverse map $\pi^{-1}$ is well defined at a point $(x,y,t)\in Y$ if and only if $y \ne 0$. So, the exceptional set $E\subseteq Y$ is given by the equation $y = 0$, and $\widehat E = \pi^{-1}(E) \subseteq \AA^2$ is given by the same equation and is isomorphic to the affine line $\AA^1$.

    The map $\widehat \varphi\colon \Ga\times\AA^2 \to \AA^2$ given by $(s,(q,y)) \mapsto (q + s, y)$ defines a $\Ga$-action on $\AA^2$ that preserves the exceptional set~$\widehat E$. However, this $\Ga$-action cannot be obtained by lifting a $\Ga$-action on~$Y$. Indeed, the three points $(1,0)$, $(\varepsilon, 0)$, and $(\varepsilon^2, 0)$, where $\varepsilon\in\KK$ is a primitive cube root of unity, are sent to the same point $(0,0,1)\in Y$ by the map $\pi$.
    At the same time, the images of the points $(1 + s,0)$, $(\varepsilon + s, 0)$, and $(\varepsilon^2 + s, 0)$ under the map $\pi$ are pairwise distinct for a general $s$. 
    This means that a general element of the considered $\Ga$-subgroup of $\Aut(\AA^2)$ does not belong to $\Aut (Y)$, and the embedding $\Aut (Y) \hookrightarrow \Aut (\AA^2, \widehat E)$ from Proposition~\ref{embed} is not an isomorphism.
\end{example}

\begin{proposition}\label{sending_down}
    Let $Y$ be a rigid non-normal affine surface and let $\pi\colon \widehat Y \to Y$ be its normalization. Then there is no non-trivial $\Ga$-action on $\widehat Y$ preserving the exceptional set $\widehat E = \pi^{-1}(E)$.
\end{proposition}

\begin{proof}
    Suppose that there is a non-trivial $\Ga$-action on $\widehat Y$ preserving~$\widehat E$, and denote by $\partial$ the locally nilpotent derivation corresponding to this action. Then the vanishing ideal $I(\widehat E)\subset \KK[\widehat Y]$ is invariant under $\partial$, and there is a non-zero function $\widehat f \in I(\widehat E)$ that lies in the kernel of $\partial$. Indeed, one can take $\widehat{f} = \partial^n(\widehat{f_0})$, where $\widehat{f_0} \in I(\widehat E)$ is any non-zero function and $n$ is the maximal non-negative integer such that $\partial^n(\widehat{f_0}) \ne 0$.

    Recall that the exceptional set $E \subseteq Y$ is the vanishing set of the conductor ideal $C_Y$ by Proposition~\ref{cond_non_normal}. It follows easily that the set $\widehat E \subseteq \widehat Y$ is also the vanishing set of the ideal $C_Y$ viewed as an ideal in~$\KK[\widehat Y]$. By Hilbert's Nullstellensatz, the ideal $I(\widehat E)$ is the radical of the conductor ideal $C_Y \subseteq \KK[\widehat Y]$, so there exists a positive integer $\ell$ such that $f = \widehat{f}^\ell \in C_Y$. Since $\widehat{f}\in \Ker \partial$, we also have $f = \widehat{f}^\ell \in \Ker \partial$. We conclude that there exists a regular function $f\in\KK[Y]$ vanishing on $E$ and lying in the kernel of~$\partial$.

    Since $f^m \in \Ker \partial$ for any positive integer $m$, the replica $f^m\partial$ of $\partial$ is again a locally nilpotent derivation of $\KK[\widehat Y]$; see~\cite{Fr06}. Any function $g\in \KK[\widehat Y]$ is regular on $Y\setminus E$ and hence is regular on $D(f) = \{y \in Y \mid f(y) \ne 0\}$, since $f$ vanishes on $E$. This implies that any $g\in\KK[\widehat Y]$ equals $\frac{h}{f^m}$ for some $h\in \KK[Y]$ and some non-negative integer $m$. Recall that the algebra $\KK[Y]$ is finitely generated over $\KK$, and let $g_1,  \ldots, g_k\in \KK[Y]$ be a set of its generators. For each $i$ from $1$ to $k$, the function $\partial g_i \in \KK[\widehat Y]$ can be written as $\frac{h_i}{f^{m_i}}$ with $h_i\in\KK[Y]$. Taking $m = \max(m_1, \ldots, m_k)$, we see that the locally nilpotent derivation $f^m\partial$ sends each of $g_1, \ldots, g_k$ to a regular function on~$Y$. Since $g_1, \ldots, g_k$ generate $\KK[Y]$ over $\KK$, the derivation $f^m\partial$ sends the whole algebra $\KK[Y]$ to $\KK[Y]$. Thus the algebra $\KK[Y]$ admits a non-zero locally nilpotent derivation $f^m\partial$, and hence the surface $Y$ admits a non-trivial $\Ga$-action. This contradicts the rigidity of~$Y$.
\end{proof}

\section{NC-completions}\label{sec:NC}

Now we are ready to adapt the proof of Theorem~\ref{th:intro-PZ} to non-normal surfaces in Proposition~\ref{prop:PZ-variation}.
We will need the following notation for its proof.

\begin{notation}
A \emph{normal-crossing completion} (an \emph{NC-completion} for short) $(\widetilde{X}, D)$ of a normal affine surface $X$ is a projective surface $\widetilde{X}$ together with a divisor $D$ such that $X = \widetilde{X}\setminus D$ and the only singularities of $D$ are nodes. For an NC-completion $(\widetilde{X}, D)$, the associated \emph{weighted dual graph} $\Gamma(D)$ is the weighted graph whose vertices are in bijection with the irreducible components of $D$ and whose edges are in bijection with the nodes of $D$. The loops of $\Gamma(D)$ at a vertex $C$ are in bijection with the self-intersection points of the component $C$ of $D$, and the edges joining two distinct vertices $C_1$ and $C_2$ of $\Gamma(D)$ are in bijection with the points of $C_1 \cap C_2$. The weight of $C$ in $\Gamma(D)$ is the self-intersection index~$C^2$ of the component $C$ in $\widetilde X$.
The vertices of $\Gamma(D)$ corresponding to the rational components of $D$ are also called \emph{rational}. 

A vertex of a weighted graph  $\Gamma$ of valence $\ge3$ is called a \emph{branching vertex}, and a vertex of valence 1 is called a \emph{tip}.
An \emph{extremal linear segment} of $\Gamma$ is a maximal connected subgraph of $\Gamma$ without branching vertices that contains a tip. 
 It is called \emph{admissible} if all its weights are at most $-2$.
For more details and results on these notions, we refer to~\cite{PeZa26}.
\end{notation}

Theorem~\ref{th:intro-PZ} is based on the following rigidity criterion.

\begin{theorem}[{\cite[Theorem~1.0.3(b)]{PeZa26}}]\label{th:PZ-NC} 
    Let $Y$ be a normal affine surface over $\KK$ and  $(X,D)$ be a minimal completion of $Y$ by a normal-crossing divisor. 
Then $\Aut^{\circ} (Y)$ is an algebraic group if and only if every extremal linear segment of the weighted dual graph $\Gamma(D)$ is admissible.
\end{theorem}

\begin{proposition}\label{prop:PZ-variation}
    Let $X$ be a normal affine surface and let $E\subset X$ be a proper closed subset.
    If the subgroup $\Aut(X,E)\subset\Aut(X)$ of automorphisms of $X$ that send $E$ to itself does not contain a $\Ga$-subgroup, 
    then its neutral component $\Aut^\circ(X,E)$ is algebraic.
\end{proposition}
\begin{proof}
    We adapt the proof of Theorem~\ref{th:intro-PZ} by including $E$ into the boundary divisor of an NC-completion as a ``component'' that is preserved by birational transformations of the completion.
    We follow the notation of \cite{PeZa26} for the rest of the proof.

    Consider an NC-completion $(\widetilde X, D)$ of $X$ and let $\overline{E}$ be the closure of $E$ in $\widetilde X$.
    We may assume that the intersections of $\overline{E}$ with $D$ are transversal.
    
    Let us extend the dual graph $\Gamma(D)$ by adding a vertex $v$ corresponding to $\overline{E}$ and edges corresponding to its intersection points with the components of $D$. 
    We mark $v$ as a \emph{non-rational} vertex in the sense of \cite[Definition~3.1.1]{PeZa26}, meaning that it is forbidden to contract $v$; the weight of $v$ is then irrelevant. 
    We denote the resulting graph by $\widetilde \Gamma$. 
    
    Let us show that $\widetilde\Gamma$ does not contain a non-admissible extremal segment. 
    Assume the contrary; then there exists a birational transformation $\psi\colon \widetilde \Gamma\dashrightarrow\widetilde \Gamma'$ such that  $\widetilde \Gamma'$ contains a tip $0$-vertex $w$; see \cite[Proposition~5.0.4]{PeZa26} and its proof. 
    
    Since $v$ is non-rational, $\psi$ induces a birational transformation $\tilde\psi\colon(\widetilde X, D)\dashrightarrow (\widetilde X', D')$. In particular, $\widetilde\Gamma'$ is obtained from $\Gamma(D')$ by adding $v$ and the appropriate edges.

    Since $X\embed \widetilde X'$ is affine, a general fiber of the $\PP^1$-fibration $\pi\colon \widetilde X'\to B$ induced by the zero curve $w$, as in the proof of \cite[Proposition~5.0.4]{PeZa26}, meets $D'$ at one point and hence is disjoint from $E$. 
    Thus, $E$ is contained in a finite number of fibers of $\pi$, and there exists a non-trivial $\Ga$-action on $X$ along $\pi$ that fixes $E$ pointwise. 
    This contradicts the assumption that $\Aut(X,E)$ contains no $\Ga$-subgroups. 
    
    So, $\widetilde\Gamma$ does not contain a non-admissible extremal segment, and by \cite[Proposition~4.2.5.3]{PeZa26} any birational transformation of $\widetilde \Gamma$ is inner.
    Since $\Aut(X,E)$ acts regularly on $E$, we have a natural map $\Aut(X,E)\to\Inn(\widetilde{\Gamma})$; cf.~\cite[Lemma~3.5.8]{PeZa26}.
    Thus, $\Aut^\circ(X,E)$ is a subgroup of $\Inn^\circ(\widetilde X,D)$, which is algebraic and regular on $\widetilde X$ by \cite[Proposition~2.3.3]{PeZa26}.
\end{proof}

\section{Proof of the main result}\label{sec:main}

\begin{proof}[Proof of Theorem~\ref{main_result}]
    The case of a normal surface $Y$ is covered by Theorem~\ref{th:intro-PZ}.
    Assume that $Y$ is not normal, and let $\pi\colon \widehat Y \to Y$ be its normalization.
    
    In the notation of Proposition~\ref{sending_down}, the subgroup $\Aut(\widehat Y,\widehat E)$ of automorphisms of $\widehat Y$ preserving $\widehat E$ does not contain a $\Ga$-subgroup.
    By Proposition~\ref{prop:PZ-variation}, applied to $X = \widehat Y$ and $E = \widehat E$, its neutral component $\Aut^\circ(\widehat Y,\widehat E)$ is an algebraic group.
    Finally, by Proposition~\ref{embed}, the neutral component $\Aut^\circ(Y)$ is embedded as a subgroup of $\Aut^\circ(\widehat Y,\widehat E)$. So, $\Aut^\circ(Y)$ is an affine algebraic group. Moreover, since $Y$ is rigid, $\Aut^\circ(Y)$ contains no $\Ga$-subgroups and hence no unipotent elements. Therefore, $\Aut^\circ(Y)$ is an algebraic torus.
\end{proof}

\bibliographystyle{plainurl}
\bibliography{references} 

%\nocite{*} % remove \nocite{*} before submission

\end{document}